\documentclass[12pt]{amsart}  

\usepackage[latin1]{inputenc}
\usepackage{amsmath}
\usepackage{amsfonts}
\usepackage{amssymb}
\usepackage{stmaryrd}
\usepackage{latexsym}
\usepackage{graphicx}
\usepackage{subfigure}
\usepackage{color}
\usepackage{hyperref}
\usepackage{verbatim}
\usepackage[all]{xy}
\usepackage{graphics}
\usepackage{pdfsync}

\theoremstyle{plain}
\newtheorem{theorem}{Theorem}[section]
\newtheorem{proposition}[theorem]{Proposition}

\newtheorem{lemma}[theorem]{Lemma}
\newtheorem{corollary}[theorem]{Corollary}

\newtheorem{definition}[theorem]{Definition}
\newtheorem{example}[theorem]{Example}

\theoremstyle{remark}
\newtheorem{remark}[theorem]{Remark}

\numberwithin{equation}{section}
\renewcommand{\comment}[1]{\vspace{5 mm}\par \noindent
\framebox{\begin{minipage}[c]{0.95 \textwidth}
 #1 \end{minipage}}\vspace{5 mm}\par}

\renewcommand{\comment}[1]{}

\newcommand{\new}[1]{\ensuremath{\blacktriangleright}#1\ensuremath{\blacktriangleleft}}
\renewcommand{\new}[1]{#1}

\newcommand{\cnew}[1]{\ensuremath{\blacktriangleright}#1\ensuremath{\blacktriangleleft}}
\renewcommand{\cnew}[1]{#1}

\newcommand{\bC}{\mathbb{C}}

\newcommand{\Ly}{\mathcal{L}_Y}
\newcommand{\Lc}{\mathcal{L}_{c}}

\newcommand{\Sym}{\ensuremath{\operatorname{Sym}}}

\newcommand{\Qsym}{\ensuremath{\operatorname{QSym}}}
\newcommand{\qs}{{\mathcal{S}}}		

\newcommand\partitionof[1]{\widetilde{#1}}
\newcommand\reverse[1]{{#1}^{r}}

\newcommand{\cont}{\mathrm{cont}} 
\newcommand{\rT}{T} 
\newcommand{\rtau}{{\tau}} 

\newcommand{\suchthat}{\;|\;}
\newcommand{\cskew}{{/\!\!/}}
\newcommand{\rhoc}{\rho}

\newcommand{\spam}{\operatorname{span}}

\newcommand{\rect}{\mathrm{rect}}
\newcommand{\stan}{\mathrm{stan}}

\newcommand{\SRT}{\ensuremath{\operatorname{SRT}}}
\newcommand{\SSRT}{\ensuremath{\operatorname{SSRT}}}
\newcommand{\LRRTL}{\ensuremath{\operatorname{LRRT}^\mathfrak{l}}}
\newcommand{\LRRTR}{\ensuremath{\operatorname{LRRT}^\mathfrak{r}}}

\newcommand{\SRCT}{\ensuremath{\operatorname{SRCT}}}
\newcommand{\SSRCT}{\ensuremath{\operatorname{SSRCT}}}
\newcommand{\LRRCTL}{\ensuremath{\operatorname{LRRCT}^\mathfrak{l}}}
\newcommand{\LRRCTR}{\ensuremath{\operatorname{LRRCT}^\mathfrak{r}}}

\newlength\cellsize 
\savebox2{%
\begin{picture}(15,15)
\put(0,0){\line(1,0){15}}
\put(0,0){\line(0,1){15}}
\put(15,0){\line(0,1){15}}
\put(0,15){\line(1,0){15}}
\end{picture}}
\newcommand\cellify[1]{\def\thearg{#1}\def\nothing{}%
\ifx\thearg\nothing
\vrule width0pt height\cellsize depth0pt\else
\hbox to 0pt{\usebox2\hss}\fi%
\vbox to 15\unitlength{
\vss
\hbox to 15\unitlength{\hss$#1$\hss}
\vss}}
\newcommand\tableau[1]{\vtop{\let\\=\cr
\setlength\baselineskip{-16000pt}
\setlength\lineskiplimit{16000pt}
\setlength\lineskip{0pt}
\halign{&\cellify{##}\cr#1\crcr}}}
\savebox3{%
\begin{picture}(15,15)
\put(0,0){\line(1,0){15}}
\put(0,0){\line(0,1){15}}
\put(15,0){\line(0,1){15}}
\put(0,15){\line(1,0){15}}
\end{picture}}
\newcommand\expath[1]{%
\hbox to 0pt{\usebox3\hss}%
\vbox to 15\unitlength{
\vss
\hbox to 15\unitlength{\hss$#1$\hss}
\vss}}
\newcommand\bas[1]{\omit \vbox to \cellsize{ \vss \hbox to \cellsize{\hss$#1$\hss} \vss}}

\newcommand{\gst}{\bullet}

\begin{document}

\title[LR rules for symmetric skew quasisymmetric Schur functions]{Littlewood-Richardson rules for symmetric skew quasisymmetric Schur functions}
\author{Christine Bessenrodt}
\address{Institut f\"ur Algebra, Zahlentheorie und Diskrete Mathematik, Leibniz Universit\"at Hannover, Hannover, D-30167, Germany}
\email{bessen@math.uni-hannover.de}

\author{Vasu Tewari}
\address{Department of Mathematics, University of British Columbia, Vancouver, BC V6T 1Z2, Canada}
\email{vasu@math.ubc.ca}

\author{Stephanie van Willigenburg}
\address{Department of Mathematics, University of British Columbia, Vancouver, BC V6T 1Z2, Canada}
\email{steph@math.ubc.ca}
\thanks{
The second and third authors were supported in part by the National Sciences and Engineering Research Council of Canada. The third author was supported in part by the Alexander von Humboldt Foundation.}
\begin{abstract}
The classical Littlewood-Richardson rule is a rule for computing coefficients in many areas, and comes in many guises. In this paper
we prove two Littlewood-Richardson rules for symmetric skew quasisymmetric Schur functions that are analogous to the famed version of the classical Littlewood-Richardson rule involving Yamanouchi words. Furthermore, both our rules contain this classical Littlewood-Richardson rule as a special case. We then apply our rules
to combinatorially classify symmetric skew quasisymmetric Schur functions. This answers affirmatively a conjecture of Bessenrodt, Luoto and van Willigenburg.
\end{abstract}

\maketitle

\section{Introduction}
An important linear basis for the \new{ring} of \new{commutative} symmetric functions, $\Sym$, is the basis of Schur functions, denoted by $s_{\lambda}$ where $\lambda$ runs over all partitions. Since $\Sym$ has various interpretations such as the representation ring of the symmetric group, or the cohomology ring of the Grassmannian, it is natural to expect that the basis of Schur functions plays a significant role in these settings as well. Indeed, this basis captures a significant amount of the interplay between algebraic combinatorics and fields such as representation theory and algebraic geometry as is demonstrated aptly by the Littlewood-Richardson rule. This rule is a combinatorial algorithm that allows us to compute the structure coefficients $c_{\nu\mu}^{\lambda}$ occurring in the expansion of the skew Schur function $s_{\lambda/\mu}$ in the Schur basis
\begin{align*}
s_{\lambda/\mu}=\sum_{\nu}c^{\lambda}_{\nu\mu}s_{\nu}.
\end{align*}
The $c^{\lambda}_{\nu\mu}$ above are called Littlewood-Richardson coefficients. Additionally, they arise in representation theory as multiplicities of irreducible $GL_n(\bC)$ representations in the tensor product of two irreducible $GL_n(\bC)$ representations. Meanwhile in algebraic geometry, they are the structure coefficients describing the cup product of two Schubert classes. Thus, given the key role played by these coefficients in various areas of mathematics, a combinatorial rule that allows us to compute them has many applications. The Littlewood-Richardson rule was first formulated in \cite{LR-1} and was proved by Sch\"utzenberger \cite{Schutzenberger} and Thomas \cite{Thomas} independently in the 1970s. Since then, various versions of the rule have been given \cite{berenstein-zel,Gasharov,Knutson-Tao,Kreiman,Stembridge} and have contributed to a rich theory that has strong links with other combinatorial constructions such as the Robinson-Schensted algorithm, jeu de taquin and the plactic monoid \cite{FominGreene,LascouxSchutzenberger}. Arguably the most popular version involves Yamanouchi words (or lattice words) with content a partition, and this is the one that is relevant for our purposes.

A \new{ring} that contains $\Sym$ and is closely tied to it is the \new{ring} of quasisymmetric functions, $\Qsym$. Quasisymmetric functions were introduced by Gessel in \cite{Gessel} as generating functions for $P$-partitions. Since then, they have become an indispensable tool in analyzing combinatorial structures as is evident from their applications to areas as diverse as enumerating permutations \cite{GesselReutenauer}, card shuffling \cite{stanley-shuffle}, combinatorial expansions of special functions such as Macdonald polynomials \cite{HHL-1} and Kazhdan-Lusztig polynomials \cite{BilleraBrenti}, representation theory of the Hecke algebra \cite{hivert-1}, discrete geometry \cite{BilleraHsiaovanWilligenburg}, and combinatorial Hopf algebras \cite{AguiarBergeronSottile}.

Haglund, Luoto, Mason, and van Willigenburg \cite{HLMvW11} introduced a new basis for $\Qsym$ using combinatorial objects called semistandard reverse composition tableaux. This is the basis of quasisymmetric Schur functions. These functions refine the Schur functions in a natural way and also refine many of their properties.

In \cite{BLvW}, Bessenrodt, Luoto, and van Willigenburg
defined skew quasisymmetric Schur functions and showed that these functions expand nonnegatively in the basis of quasisymmetric Schur functions.
\begin{align*}
\qs_{\alpha \cskew \beta} = \sum_{\delta}C^{\alpha}_{\delta\beta}\qs_{\delta}
\end{align*}
The $C_{\delta\beta}^{\alpha}$ are called noncommutative Littlewood-Richardson coefficients. Additionally, they possess an interpretation that is similar to an interpretation of the classical Littlewood-Richardson coefficients. What is lacking though is an interpretation involving Yamanouchi words, which can be attributed to the fact that constructing Yamanouchi words with content a composition (that is not a partition) is impossible. However this situation can be remedied if we restrict our attention to \emph{symmetric} skew quasisymmetric Schur functions.

\begin{remark}The skew quasisymmetric Schur functions defined in \cite{BLvW}
might be termed ``left skew'' quasisymmetric Schur functions
as the authors collect the terms on the left side in the coproduct.
The study of the combinatorics of the
corresponding ``right skew'' quasisymmetric Schur functions has been initiated by Tewari in \cite{Tewari}, and
these functions arise from maximal chains in the right Pieri poset $\mathcal{R}_c$  introduced and studied in some detail therein.
These ``right skew'' quasisymmetric Schur functions $S_{\alpha\backslash\!\!\backslash  \gamma}$ have the following expansion:
\begin{eqnarray*}
S_{\alpha\backslash\!\!\backslash  \gamma}=\displaystyle\sum_{\beta}C_{\alpha\beta}^{\gamma}S_{\beta}\:.
\end{eqnarray*}
Just as the maximal chains in the poset $\mathcal{L}_c$ (that are equivalent to standard reverse composition tableaux) are the pertinent objects when expanding ``left skew'' quasisymmetric Schur functions, the maximal chains from $\alpha$ to $\gamma$ in the poset $\mathcal{R}_c$ allow the computation of the coefficients in the expansion above \cite[Corollary 8.11]{Tewari}.
However, unlike the cover relations in the poset $\mathcal{L}_c$, the cover relations in the poset $\mathcal{R}_c$ do not correspond to inclusion of composition diagrams in any way.
Thus, a Littlewood-Richardson rule for ``right skew'' quasisymmetric Schur functions involving Yamanouchi words that mirrors the classical Littlewood-Richardson rule or the generalizations given for the ``left skew'' quasisymmetric Schur functions in this article is not possible.
\end{remark}

Our \new{main} aim in this article is to give a Schur positive expansion for the symmetric skew quasisymmetric Schur functions, and to give a Yamanouchi interpretation for computing structure coefficients occurring in the expansion. We achieve this aim by giving a \emph{left} Littlewood-Richardson rule in Theorem~\ref{the:NCLR} and a \emph{right} Littlewood-Richardson rule in Theorem~\ref{the:NCLRR}. More specifically, we give a rule to compute the coefficient of $\qs_{\delta}$ in $\qs_{\alpha\cskew \beta}$ in the case where $\delta$ is a partition or the reverse of a partition. As an application of these rules, we obtain a combinatorial classification\new{, that is, a classification in terms of skew reverse composition shapes,} of symmetric skew quasisymmetric Schur functions in Theorem~\ref{the:uniform}. In so doing, we prove a conjecture of Bessenrodt, Luoto and van Willigenburg. More precisely, in \cite[Corollary 5.2]{BLvW}, they proved that if $\alpha\cskew \beta$ is a uniform skew reverse composition shape then $\qs_{\alpha\cskew\beta}$ is symmetric, and further conjectured \cite[Conjecture 6.1]{BLvW} that the reverse implication also holds. We establish that this is indeed the case and our two Littlewood-Richardson rules are crucial to this end.

\section{Background}\label{sec:background}
In this section, we will define the combinatorial objects that will be useful to us later. For more details on the classical notions discussed herein, the reader should refer to \cite{stanley-ec2}, while a detailed exposition on more modern constructs such as reverse composition diagrams, reverse composition tableaux, and quasisymmetric Schur functions can be found in \cite{QSbook}.
\subsection{Compositions and partitions}
Given a positive integer $n$, a \emph{composition} $\alpha=(\alpha_1,\ldots,\alpha_k)$ of $n$ is defined to be an ordered list of positive integers whose sum is $n$. If $\alpha$ is a composition of $n$, then we denote this by $\alpha \vDash n$. The $\alpha _i$ for $1\leq i \leq k$ are called the \emph{parts} of $\alpha$, while $k$ is called the \emph{length} of $\alpha$ and is denoted by $\ell(\alpha)$. The \emph{size} of a composition $\alpha$, denoted by $\vert \alpha\vert$, is the sum of its parts. We denote by $\reverse{\alpha}$ the composition obtained by reversing $\alpha$, that is, if $\alpha=(\alpha_1,\ldots,\alpha_k)$ then $\reverse{\alpha}=(\alpha_k,\ldots,\alpha_1)$. Finally, for convenience we define the empty composition $\varnothing$ to be the unique composition of size and length equalling $0$.

We depict a composition $\alpha=(\alpha _1, \ldots , \alpha _k)$ using its \emph{reverse composition diagram},
also denoted by $\alpha$, which is the array of left-justified cells with $\alpha _i$ cells in row $i$ from the top.
A cell is said to be in row $i$ and column $j$ if it is in the $i$-th row from the top and $j$-th column from the left, and is referred to by the pair~$(i,j)$.

\begin{example}\label{ex:2432} The reverse composition diagram of $\alpha = (2,4,3,2)\vDash 11$ is shown below.
$$\tableau{\ &\ \\\ &\ &\ &\ \\\ &\ &\  \\\ &\ }$$
Note also that $\reverse{\alpha}=(2,3,4,2)$.
\end{example}

A \emph{partition} $\lambda=(\lambda_1,\ldots,\lambda_k)$ of $n$ is a list of positive integers that sum to $n$ and additionally satisfy $\lambda_1\geq \cdots \geq \lambda_k$. The $\lambda_i$ for $1\leq i\leq k$ are called the \textit{parts} of $\lambda$, while $k$ is called the \emph{length} of $\lambda$, and is denoted by $\ell(\lambda)$. The \emph{size} of $\lambda$, denoted by $\vert \lambda\vert$, is the sum of its parts. If $\lambda$ is a partition of size $n$, then we denote this by $\lambda \vdash n$. As with compositions, given a partition $\lambda= (\lambda_1,\ldots, \lambda_k)$ we define $\reverse{\lambda}$ to be the ordered list of positive integers obtained by reversing $\lambda$, that is, $\reverse{\lambda}=(\lambda_k,\ldots,\lambda_1)$. Furthermore, we define the empty partition, denoted by $\varnothing$, to be the unique partition of size and length equalling $0$. Finally, note that given any composition $\alpha$, there is a corresponding partition, denoted by $\widetilde{\alpha}$, which is obtained by sorting the parts of $\alpha$ into weakly decreasing order.

We depict a partition using its \textit{Young diagram}. Given a partition $\lambda=(\lambda_1,\ldots,\lambda_k)\vdash n$, the Young diagram of $\lambda$, also denoted by $\lambda$, is the left-justified array of $n$ cells, with $\lambda_i$ cells in the $i$-th row. We will be using the English convention for Young diagrams. That is, the rows are numbered from top to bottom and the columns from left to right. We refer to the cell in the $i$-th row and $j$-th column by the ordered pair $(i,j)$. The \emph{transpose} of a partition $\lambda$, denoted by $\lambda^t$, is the partition whose Young diagram is the array of cells
\begin{align*}
\lambda^t=\{(i,j)\suchthat (j,i)\in \lambda\}.
\end{align*}
\begin{example}\label{ex:4322} Shown below on the left is the Young diagram of the partition $\lambda = (4,3,2,2)\vdash 11$. On the right is $\lambda^t$.
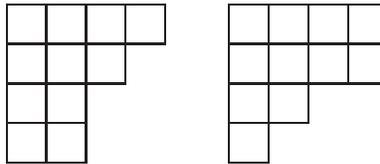
$$\tableau{\ &\ &\ &\ \\ \ &\ &\ \\ \ &\ \\ \ &\ }\qquad \tableau{\ &\ &\ &\ \\ \ &\ &\ &\ \\  \ & \ \\ \  }$$
Also, $\widetilde{\alpha}=\lambda$ where $\alpha=(2,4,3,2)$ is the composition in Example \ref{ex:2432}.
\end{example}

\subsection{Tableaux}
We will start by defining a poset structure on compositions that allows us to construct the main combinatorial object required to introduce quasisymmetric Schur functions.
Let $\alpha = (\alpha_1,\ldots, \alpha_{\ell(\alpha)})$ and $\beta$ be compositions. Define a cover relation $\lessdot _{c}$ on compositions as follows.
\begin{eqnarray*}
\alpha \lessdot_{c} \beta \Longleftrightarrow \left\lbrace \begin{array}{ll} \beta= (1,\alpha_1,\ldots,\alpha_{\ell(\alpha)}) & \\
\text{ or}\\ \beta= (\alpha_1,\ldots,\alpha_k +1,\ldots,\alpha_{\ell(\alpha)}) & \text{ and $\alpha_i\neq \alpha_k$ for all $i<k$} \end{array}\right.
\end{eqnarray*}
The \emph{reverse composition poset} $\mathcal{L}_{c}$ is the poset on the set of compositions where the partial order $<_{c}$ is obtained by taking the transitive closure of the cover relation $\lessdot_{c}$ above.
If $\beta <_{c} \alpha$ and $\beta$ is drawn in the bottom left corner of $\alpha$, then
the \textit{skew reverse composition shape} $\alpha \cskew \beta$ is
defined to be the array of cells
\begin{eqnarray*}
\alpha\cskew \beta = \{(i,j)\suchthat (i,j)\in \alpha, (i,j)\notin \beta \}.
\end{eqnarray*}
We refer to $\alpha$ as the \textit{outer shape} and to $\beta$ as the \textit{inner shape}. If the inner shape is $\varnothing$, instead of writing $\alpha \cskew \varnothing$, we just write $\alpha$ and refer to $\alpha$ as a \textit{straight shape}. The \emph{size} of the skew reverse composition shape $\alpha\cskew \beta$, denoted by $\vert \alpha\cskew\beta\vert$, is the number of cells in the skew reverse composition shape, that is, $\vert \alpha\vert-\vert\beta\vert$.

We will now define a skew reverse composition shape that will play a key role in our classification of symmetric skew quasisymmetric Schur functions in Section \ref{sec:uniform}. Let \new{$\beta <_{c} \alpha$} and suppose further that $\ell(\beta)=k$ and $\ell(\alpha)=k+\ell$. Then we call the top $\ell$ rows of $\alpha \cskew \beta$ the \emph{upper shape} of $\alpha \cskew \beta$ and the remaining bottom $k$ rows the \emph{lower shape} of $\alpha \cskew \beta$. If the upper shape of $\alpha \cskew \beta$ is a rectangle, that is, $\alpha _1 = \cdots = \alpha _\ell$, then we say that $\alpha \cskew \beta$ is \emph{uniform}.

\begin{example}\label{ex:uniform}
If $\alpha=(2,2,4,3,3)$ and $\beta=(3,1,2)$, then $\beta <_{c} \alpha$,
and $\alpha \cskew \beta$ is of uniform skew reverse composition shape as depicted
below.
$$
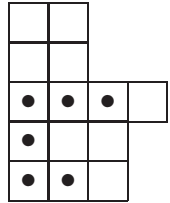
\tableau{\ & \ \\ \ & \ \\ \gst &\gst &\gst &\ \\ \gst &\ &\ \\ \gst &\gst &\ }
$$
\end{example}
Next we will define a semistandard reverse composition tableau\new{, and illustrate the definition in Example \ref{ex:SRCT3432}.}
\begin{definition}
A \emph{semistandard reverse composition tableau} (abbreviated to $\SSRCT$) $\tau$ of \emph{shape} $\alpha \cskew \beta$ is a filling
\begin{eqnarray*}
\tau: \alpha\cskew \beta \longrightarrow \mathbb{Z}^{+}
\end{eqnarray*}
that satisfies the following conditions
\begin{enumerate}
\item the entries i\new{n} each row decrease weakly from left to right,
\item the entries in the leftmost column increase strictly from top to bottom,
\item if $i< j$ and $(j,k+1) \in \alpha \cskew \beta$ and either $(i,k)\in \beta$ or $\tau(i,k) \geq \tau(j,k+1)$ then either $(i,k+1)\in \beta$ or both $(i,k+1) \in \alpha \cskew \beta $ and $\tau(i,k+1) > \tau(j,k+1)$. (This condition will be referred to as \emph{the triple condition}.)
\end{enumerate}
\end{definition}
\new{Note that the third condition above implies that no two entries in the same column can be equal.}

We will denote the set of $\SSRCT$s of shape $\alpha\cskew \beta$ by $\SSRCT(\alpha\cskew \beta)$. Given an $\SSRCT$ $\tau$, we define the \emph{content} of $\tau$, denoted by $\cont (\tau)$, to be the list of nonnegative integers
\begin{align*}
\cont (\tau)=(c_1,\ldots, c_{max})
\end{align*}
where $c_i$ equals the number of times $i$ appears in $\tau$ and $max$ is the greatest integer appearing in $\tau$.

Given a positive integer $n$, let $[n]=\{1,\ldots,n\}$. Furthermore, define $[0]$ to be the empty set. Then, a \emph{standard reverse composition tableau} (abbreviated to SRCT) is an $\SSRCT$ in which the filling is a bijection $\tau: \alpha\cskew\beta \longrightarrow [|\alpha\cskew \beta|]$. That is, in an SRCT each number from the set $\{1,2,\ldots, |\alpha\cskew\beta|\}$ appears exactly once. We denote the set of all $\SRCT$s of shape $\alpha\cskew \beta$ by $\SRCT(\alpha\cskew\beta)$.

A particular $\SRCT$ that will be of importance to us is the canonical reverse composition tableau. Given a composition $\alpha = (\alpha _1 , \ldots , \alpha _k)$, the \emph{canonical reverse composition tableau} of \emph{shape} $\alpha$, denoted by $\rtau _\alpha$, is constructed as follows. In the reverse composition diagram of $\alpha$, the $i$-th row for $1\leq i\leq k$ is filled with the consecutive positive integers from $\displaystyle\sum_{j=1}^{i}\alpha_j$ down to $\left(1+\displaystyle\sum_{j=1}^{i-1}\alpha_j\right)$ in that order from left to right.

\begin{example}\label{ex:SRCT3432}
Shown below, on the left, is an $\SSRCT$ of shape $(2,4,2,3)\cskew (1,2)$ and, on the right, is \new{the canonical reverse composition tableau} $\rtau _{(3,4,3,2)}$.
$$\rtau =
\tableau{5&3\\ 8 &8&6&3\\\gst  &2\\\gst &\gst & 9}\qquad
\rtau _{(3,4,3,2)}=
\tableau{3&2&1\\7&6&5&4\\10&9&8\\12&11}$$
The $\SSRCT$ on the left has content $(0,1,2,0,1,1,0,2,1)$.
\end{example}

Now we will define \new{analogues of $\Lc$ and $\SSRCT$s that are} central to the classical theory of symmetric functions, \new{namely, Young's lattice and semistandard reverse tableaux.}
Let $\mu = (\mu_1,\ldots, \mu_{\ell (\mu)})$ and $\lambda$ be partitions. Define a cover relation $\lessdot_{Y}$ on partitions as follows.
\begin{eqnarray*}
\mu \lessdot_{Y} \lambda \Longleftrightarrow \left\lbrace \begin{array}{ll} \lambda = (\mu_1,\ldots,\mu_{\ell (\mu)},1) & \text{ or}\\ \lambda= (\mu_1,\ldots,\mu_k +1,\ldots,\mu_{\ell(\mu)}) & \text{ and $\mu_i\neq \mu_k$ for all $i<k$} \end{array}\right.
\end{eqnarray*}
The poset on the set of partitions obtained by taking the transitive closure of the cover relation $\lessdot_{Y}$ is well-known as \emph{Young's lattice}, and we will denote it by $\mathcal{L}_{Y}$. We will denote the order relation on $\mathcal{L}_{Y}$ by $<_{Y}$.
If $\mu <_{Y} \lambda$, the \textit{skew shape} $\lambda/\mu$ is defined to be the array of cells
\begin{eqnarray*}
\lambda/\mu = \{(i,j)\suchthat (i,j)\in \lambda, (i,j)\notin \mu \}
\end{eqnarray*}
where $\mu$ is drawn in the top left corner of $\lambda$.
We refer to $\lambda$ as the \textit{outer shape} and to $\mu$ as the \textit{inner shape}. If the inner shape is $\varnothing$, instead of writing $\lambda/\varnothing$, we just write $\lambda$ and call $\lambda$ a \textit{straight shape}. The \emph{size} of a skew shape~$\lambda/ \mu$, denoted by $\vert \lambda/\mu\vert$, is the number of cells in the skew shape, which is $\vert \lambda\vert-\vert\mu\vert$.

\begin{definition}
A \emph{semistandard reverse tableau} (abbreviated to $\SSRT$) $T$ of \emph{shape} $\lambda/\mu$ is a filling
\begin{eqnarray*}
T: \lambda/\mu \longrightarrow \mathbb{Z}^{+}
\end{eqnarray*}
that satisfies the following conditions
\begin{enumerate}
\item the entries in each row decrease weakly from left to right,
\item the entries in each column decrease strictly from top to bottom.
\end{enumerate}
\end{definition}
We will denote the set of $\SSRT$s of shape $\lambda/ \mu$ by $\SSRT(\lambda/\mu)$. Exactly as we did with $\SSRCT$s, we define the \emph{content} of an $\SSRT$ $T$, denoted by $\cont (T)$, to be the list of nonnegative integers
\begin{align*}
\cont (\tau)=(c_1,\ldots, c_{max})
\end{align*}
where $c_i$ equals the number of times $i$ appears in $\tau$ and $max$ is the greatest integer appearing in $\tau$.

A \emph{standard reverse tableau} (abbreviated to $\SRT$) is an $\SSRT$ in which the filling is a bijection $T: \lambda/\mu \longrightarrow [|\lambda/\mu|]$, that is, each number from the set $\{1,2,\ldots, |\lambda/\mu|\}$ appears exactly once.
The set of all $\SRT$s of shape $\lambda/\mu$ is denoted by $\SRT(\lambda/\mu)$.

As in the case of $\SRCT$s, there is a distinguished tableau associated with any partition. Given a partition $\lambda = (\lambda _1 , \ldots , \lambda _k)$, the \emph{canonical reverse tableau} of \emph{shape} $\lambda$, denoted by $T _\lambda$, is constructed as follows. In the Young diagram of $\lambda$, the $i$-th row for $1\leq i\leq k$ is filled with the consecutive positive integers from $\displaystyle\sum_{j=i}^{k}\lambda_{j}$ down to $\left(1+\displaystyle\sum_{j=i+1}^{k}\lambda_{j}\right)$ in that order from left to right.
\begin{example}\label{ex:SSRT4332}
Shown below, on the left, is an $\SSRT$ of shape $(4,3,2,2)/(2,1)$ and, on the right, is \new{the canonical reverse tableau} $T _{(4,3,3,2)}$.
$$\tableau{\gst & \gst &9 &3 \\ \gst & 8&6 \\ 8&3\\5&2 }\qquad \tableau{12 &11&10 &9\\ 8&7&6\\5&4&3\\2&1}$$
\end{example}

It might seem that the two types of tableaux introduced in this subsection, $\SSRCT$s and $\SSRT$s, are unrelated. To the contrary, there is a bijection between appropriate sets of $\SSRCT$s and $\SSRT$s. We will describe this bijection next.

Let $\SSRCT(-\cskew \beta)$ denote the set of all $\SSRCT$s with inner shape $\beta$, and $\SSRT(-/\partitionof{\beta})$ denote the set of all $\SSRT$s with inner shape $\partitionof{\beta}$. Then
\begin{equation*}\label{eq:rhocalpha}
\rhoc _{\beta}: \SSRCT(-\cskew \beta) \to \SSRT(- / \partitionof{\beta})
\end{equation*}is defined as follows \cite[Chapter 4]{QSbook}\new{, which generalizes the map for semistandard skyline fillings \cite{mason-1}.}
Given $\rtau\in \SSRCT(-\cskew \beta)$, obtain $\rhoc _\beta(\rtau)$ by writing the entries in each column in decreasing order from top to bottom and top-justifying these new columns on the inner shape $\partitionof{\beta}$ (that might be empty).

The inverse map
\begin{equation*}\label{eq:rhocinversealpha}
\rhoc ^{-1}_\beta : \SSRT(-/\partitionof{\beta})\to \SSRCT(-\cskew \beta)
\end{equation*}is also straightforward to define.

Given $\rT\in \SSRT(-/\partitionof{\beta})$,
\begin{enumerate}
\item take the set of $i$ entries in the leftmost column of $\rT$ and write them in increasing order in rows $1, 2, \ldots, i$ above the inner shape with a cell in position $(i+1, 1)$ but not $(i,1)$ to form the leftmost column of $\rtau$,
\item take the set of entries in column 2 in decreasing order and place them in the row with the smallest index so that either
\begin{itemize}
\item the cell to the immediate left of the number being placed is filled and the row entries weakly decrease when read from left to right
\item the cell to the immediate left of the number being placed belongs to the inner shape,
\end{itemize}
\item repeat the previous step with the set of entries in column $k$ for $k= 3, \ldots , \partitionof{\beta} _1$.
\end{enumerate}

\begin{example}\label{ex:rhocinverse}
Let $\beta=(2,3,2)$.
$$\tableau{\gst & \gst & \gst & 8 &7\\ \gst & \gst & 6&4 \\\gst & \gst & 3\\5&5&2\\4\\1}\quad \mathrel{\mathop{\rightleftarrows}^{\mathrm{\rho_{(2,3,2)}^{-1}}}_{\mathrm{\rho_{(2,3,2)}}}} \quad \tableau{1\\4\\5&5&3\\ \gst & \gst & 6&4\\ \gst & \gst & \gst & 8&7\\ \gst & \gst & 2}$$
\end{example}

Finally, recall that we can \emph{standardize} our tableaux of either sort by taking the, say, $\nu _1$ 1s from right to left and relabelling them $1, 2, 3, \ldots , \nu_1$, then the, say, $\nu _2$ 2s from right to left and relabelling them $\nu_1 +1 , \ldots , \nu_1+\nu _2$ etc. The standardization of an $\SSRCT$ $\tau$ (respectively $\SSRT$ $T$) will be denoted by $\stan (\tau)$ (respectively $\stan (T)$); observe that standardization gives an SRCT (respectively SRT), and that it commutes with the map $\rho_\beta$.
We will do an example in the case of $\SSRCT$s next.
\begin{example}\label{ex:standardize}
Shown below are an $\SSRCT$ of shape $(9,3,2,4)\cskew (4,1,1,3)$ (left) and its standardization (right).
$$\tableau{\bullet&\bullet&\bullet&\bullet&3&3&1&1&1\\
\bullet&2&1\\
\bullet&1\\
\bullet&\bullet&\bullet&2}
\quad\quad \tableau{\bullet & \bullet &\bullet &\bullet &9&8&3&2&1\\ \bullet &7&4 \\ \bullet &5 \\ \bullet &\bullet &\bullet & 6}
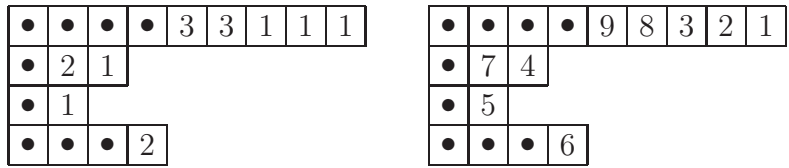
$$
\end{example}

\subsection{Rectification of tableaux}
Another tool we will need is the \emph{rectification} of an $\SRCT$, which is inserting the entries of a column of the $\SRCT$ taken in increasing order (and columns taken from left to right) into an empty $\SSRCT$ using the following insertion process \cite[Procedure 3.3]{mason-1}. Suppose we start with an $\SSRCT$ $\rtau$ whose longest row has length $r$.
To insert a positive integer $k_1$, the result being denoted $k_1\to
\rtau$,
scan column positions starting with the top position in column $j
= r+1$.
\begin{enumerate}
\item If the current position is empty and at the end of a row of
length $j-1$, and $k_1$ is weakly less than
the last entry in the row, then place $k_1$ in this empty position
and stop. Otherwise, if the position
is nonempty and contains $k_2<k_1$ and $k_1$ is weakly less than the
entry to the immediate left of $k_2$,
let $k_1$ bump $k_2$, that is, swap $k_2$ and $k_1$.
\item Using the possibly new $k_i$ value, continue scanning successive
positions in the column top to bottom as in the previous step, bumping
whenever possible, and then continue scanning at the top of the next
column to the left. (Decrement $j$.)
\item If an element is bumped into the leftmost column, then create a new
row containing one cell to contain the element,
placing the row such that the leftmost column is strictly increasing top
to bottom, and stop.
\end{enumerate}
\begin{example}\label{ex:MRSK}
$$5 \to \quad \tableau{
   1 & 1 \\
   3 & 2 & 2 & 2 \\
   6 & 5 & 4  \\
   7 & 4 & 3   \\
}
 \qquad =  \qquad
\tableau{
   1  & 1\\
   {2}               \\
   3  & {3} & 2 & 2 \\
   6        & 5 & {5}  \\
   7        & 4 & {4}   \\
} $$\end{example}
Given an $\SRCT$ $\tau$, its rectification is denoted by $\rect (\tau)$.
\begin{example}\label{ex:rectificationSRCT}
Let $\tau$ be the $\SRCT$ shown below.
$$
\tableau{ 3 & 2\\ \gst & \gst & \gst & 6\\ \gst & 1\\ \gst & \gst & 5 & 4}
$$
Then $\rect(\tau)$ is computed by inserting the integers $3,1,2,5,4,6$ in that order using the insertion process outlined above, starting from the empty $\SSRCT$. We obtain the following sequence of $\SSRCT$s, with the final tableau being $\rect(\tau)$.
\begin{align*}
\varnothing \rightarrow \tableau{3} \rightarrow \tableau{3  &1} \rightarrow \tableau{1\\3& 2}\rightarrow\tableau{1\\3& 2\\5}\rightarrow\tableau{1\\3 & 2\\5&4}\rightarrow\tableau{1\\3&2\\5&4\\6}
\end{align*}
Notice that $\rect(\tau)=\tau_{(1,2,2,1)}$, or in words, $\tau$ rectifies to $\tau_{(1,2,2,1)}$.
\end{example}

This algorithm is analogous to the \emph{rectification} of an $\SRT$ \cite[Section 2.6]{QSbook}, which is inserting the entries of a column of the $\SRT$ taken in increasing order (and columns taken from left to right) into an empty $\SSRT$. The following insertion process is used, which inserts a positive integer $k_1$ into an $\SSRT$ $\rT$ and is denoted by $\rT\leftarrow k_1$.
\begin{enumerate}
\item If $k_1$ is less than or equal to the last entry in row 1, place it at the end of the row, else
\item find the leftmost entry in that row strictly smaller than $k_1$, say $k_2$, then
\item replace $k_2$ by $k_1$, that is, $k_1$ {bumps} $k_2$.
\item Repeat the previous steps with $k_2$ and row 2, $k_3$ and row 3, etc.
\end{enumerate}

\begin{example}\label{ex:rschentsted}
$$\tableau{7 & 5 & 4 & 2\\6 & 4 & 3  \\3 & 2 & 2 \\1 & 1
}\quad  \gets 5  \qquad
=  \qquad
\tableau{
 7        & 5 &  {5} & 2\\ 6        & 4 &  {4}  \\3 &  {3}  & 2 \\{2} & 1 \\{1}
}$$
\end{example}
Given an $\SRT$ $T$, we denote its rectification by $\rect (T)$.
\begin{example}\label{ex:rectificationSRT}Let $T$ be the $\SRT$ shown below.
$$\tableau{\gst &\gst &\gst & 6\\ \gst &\gst &5 &4 \\\gst & 2 \\3&1 }$$
To compute the rectification of $T$, we insert the integers $3,1,2,5,4,6$ in that order into an empty $\SSRT$. We get the following sequence of $\SSRT$s.
\begin{align*}
\varnothing \rightarrow
\tableau{3}
\rightarrow
\tableau{3&1}\rightarrow
\tableau{3&2\\1}\rightarrow
\tableau{5&2\\3\\1}\rightarrow
\tableau{5&4\\3&2\\1}\rightarrow
\tableau{6&4\\5&2\\3\\1}
\end{align*}
The final $\SRT$ obtained above is $\rect(T)$.
\end{example}

\subsection{Symmetric functions and Schur functions}
Consider $\bC [[x_1 , x_2 , x_3 , \ldots ]]$, the graded \new{ring} of formal power series \new{of bounded total degree} in the commuting indeterminates $x_1,x_2,\ldots$ graded by total monomial degree. An important \new{subring} of it is the \new{ring} of symmetric functions. We will define it by giving a basis for it.

\begin{definition}\label{def:Msymmetric}
Let $\lambda=(\lambda_1,\ldots,\lambda_k)\vdash n$ be a partition. Then the
\emph{monomial symmetric function} $m_{\lambda}$ is defined by
\[m_{\lambda}=\sum_{\substack{(i_1,\ldots,i_k)\\i_1<\cdots <i_k}}  \; \sum_{\substack{\alpha\vDash n\\\widetilde{\alpha}=\lambda}} x_{i_1}^{\alpha_1}\cdots x_{i_k}^{\alpha_k}.\]
Furthermore, we set $m_\varnothing=1$.
\end{definition}
\begin{example}\label{ex:mlambda}
We have
$m_{(2,1)} = x_1^2x_2^1 +x_1^1x_2^2+x_1^2x_3^1+x_1^1x_3^2+\cdots $.
\end{example}
We can now define the \new{\emph{ring of symmetric functions}}, $\Sym$, which is a graded \new{subring} of $\bC [[x_1 , x_2 , x_3 , \ldots ]]$, by $$\Sym = \bigoplus _{n\geq 0} \Sym ^n$$ where $$\Sym ^n = \spam\{ m_\lambda\suchthat \lambda \vdash n\}.$$

Arguably the most important basis of $\Sym$ is the basis of Schur functions for the reasons given in the introduction. To describe this basis, we need to associate monomials with tableaux. Given $T\in \SSRT(\lambda/\mu)$, associate a monomial $\mathbf{x}^T$ to it as follows.
\begin{align*}
\mathbf{x}^T=\prod_{(i,j)\in \lambda/\mu}x_{T(i,j)}
\end{align*}
This given, we define the \emph{skew Schur function} indexed by the skew
shape~$\lambda/\mu$, denoted by $s_{\lambda/\mu}$, to be
\begin{align*}
s_{\lambda/\mu}=\sum_{T\in \SSRT(\lambda/\mu)}\mathbf{x}^T.
\end{align*}
If $\mu=\varnothing$ in the above definition, instead of writing $s_{\lambda/\varnothing}$,
we write $s_{\lambda}$ and call it a \emph{Schur function}.
For
an empty diagram, we set $s_\varnothing =1$.
Though not evident from the above definition,
\new{$s_{\lambda/\mu}$} is a symmetric function. Furthermore, the elements of the set $\{s_{\lambda}\suchthat\lambda\vdash n\}$ form a basis of $\Sym^{n}$ for any nonnegative integer~$n$. This implies that we can expand a skew Schur function as a $\bC$-linear combination of Schur functions. In fact, any skew Schur function is a linear combination of Schur functions with nonnegative integral coefficients:
\begin{align*}
s_{\lambda/\mu}=\sum_{\nu \vdash |\lambda/\mu|}c_{\nu\mu}^{\lambda}s_{\nu}.
\end{align*}
The structure coefficients $c_{\nu\mu}^{\lambda}$ above are called {Littlewood-Richardson coefficients}, and there are many combinatorial algorithms for computing them, two of which we will encounter later.

The skew Schur functions belong to a special class of symmetric functions, called \emph{Schur positive functions}, which are defined to be those symmetric functions that expand as a nonnegative linear combination of Schur functions.

\subsection{Quasisymmetric functions and quasisymmetric Schur functions}\label{subsec:Qsym} In this subsection we will define another \new{subring} of $\bC [[x_1 , x_2 , x_3 , \ldots ]]$ that contains $\Sym$\new{}. We will then proceed to describe a distinguished basis of this \new{subring}, called the basis of quasisymmetric Schur functions. This basis will be central to our investigations later on.
\begin{definition}\label{def:Mquasisymmetric}
Let $\alpha=(\alpha_1,\ldots,\alpha_k)\vDash n$ be a composition. Then the
\emph{monomial quasisymmetric function} $M_{\alpha}$ is defined by
\[M_{\alpha}=\sum_{\substack{(i_1,\ldots,i_k)\\i_1<\cdots <i_k}}
x_{i_1}^{\alpha_1}\cdots x_{i_k}^{\alpha_k}.\]
Furthermore, we set $M_\varnothing=1$.
\end{definition}

\begin{example}\label{ex:Malpha}
We have
$M_{(1,2)} = x_1^1x_2^2 +x_1^1x_3^2+x_1^1x_4^2+x_2^1x_3^2+\cdots $.
\end{example}

We can now define the \new{\emph{ring of quasisymmetric functions}}, $\Qsym$, which is a graded \new{subring} of $\bC [[x_1 , x_2 , x_3 , \ldots ]]$, by $$\Qsym = \bigoplus _{n\geq 0} \Qsym ^n$$ where $$\Qsym ^n = \spam\{ M_\alpha\suchthat \alpha \vDash n\}.$$

Just as we did in the case of $\SSRT$s, we can associate a monomial with an $\SSRCT$. Given $\tau \in \SSRCT(\alpha\cskew \beta)$, we associate to it a monomial $\mathbf{x}^{\tau}$ as follows.
\begin{align*}
\mathbf{x}^{\tau}=\prod_{(i,j)\in \alpha\cskew\beta} x_{\tau(i,j)}
\end{align*}
Then we can define the \emph{skew quasisymmetric Schur function} indexed by the skew reverse composition shape~$\alpha\cskew \beta$, denoted by $\qs_{\alpha\cskew\beta}$, to be
\begin{align*}
\qs_{\alpha\cskew\beta}=\sum_{\tau\in \SSRCT(\alpha\cskew\beta)}\mathbf{x}^{\tau}.
\end{align*}
If $\beta=\varnothing$, instead of writing $\qs_{\alpha\cskew \varnothing}$, we write $\qs_{\alpha}$ and call this the \emph{quasisymmetric Schur function} indexed by $\alpha$.
For an empty diagram, we set $\qs_\varnothing =1$.
An important property of quasisymmetric Schur functions is that
\begin{align*}
\Qsym ^n = \spam\{ \qs_\alpha\suchthat \alpha \vDash n\}.
\end{align*}
The original motivation for naming these functions quasisymmetric Schur functions is the following equality that shows that quasisymmetric Schur functions refine Schur functions in a natural way \new{\cite[Section 5]{HLMvW11}.}
$$s_{\lambda} = \sum _{\partitionof{\alpha} = \lambda} \qs _\alpha$$

\begin{example}\label{ex:Salpha} We compute the following expansion
$$\qs _{(1,2)} = x_1^1x_2^2+x_2^1x_3^2+x_1^1x_3^2+x_1^1x_2^1x_3^1+x_2^1x_3^1x_4^1+\cdots$$
from the monomials associated with the following $\SSRCT$s.
$$\tableau{1\\2&2}
\quad
\tableau{2\\3&3}
\quad
\tableau{1\\3&3}
\quad
\tableau{1\\3& 2}
\quad
\tableau{2\\4&3}\cdots$$
\end{example}

\medskip

In \cite{BLvW}, it was shown that skew quasisymmetric Schur functions expand as a nonnegative integer linear combination of quasisymmetric Schur functions,
and the expansion was given precisely as a generalization of the classical Littlewood-Richardson rule. We now recall this main result \cite[Theorem 3.5, Equation (3.5)]{BLvW},
which gives a \emph{noncommutative Littlewood-Richardson rule}.

\begin{theorem}
\label{BLvW-main}
Given a skew reverse composition shape~$\alpha\cskew \beta$, we have
\begin{align*}
\qs_{\alpha\cskew\beta}=\sum_{\delta\vDash |\alpha\cskew\beta|}C_{\delta\beta}^{\alpha}\qs_{\delta},
\end{align*}
where the structure coefficients $C_{\delta\beta}^{\alpha}$, called \emph{noncommutative Littlewood-Richardson coefficients}, are nonnegative integers described combinatorially by
\begin{align}\label{eq:nclr definition}
C_{\delta\beta}^{\alpha}=|\{\tau\in \SRCT(\alpha\cskew\beta)\suchthat \rect(\tau)=\tau_{\delta}\}|.
\end{align}
\end{theorem}

\begin{example}\label{ex:noncommutativeLR example}
We compute that $C_{(1,2)(3,2)}^{(1,4,3)}=2$ by verifying that only the two $\SRCT$s shown below rectify to $\tau_{(1,2)}$.
\begin{align*}
\tableau{1\\ \gst & \gst &\gst &2\\ \gst & \gst &3}
\quad
\tableau{3\\ \gst & \gst & \gst & 2\\ \gst & \gst & 1}
\end{align*}
Thus, $\qs_{(1,4,3)\cskew (3,2)}= \cdots + 2\qs_{(1,2)}+\cdots$.
\end{example}
We also say a function is \emph{quasisymmetric Schur positive} if it can be written as a nonnegative sum of quasisymmetric Schur functions.

Next, we will establish some small but useful results relating symmetric and quasisymmetric functions.
\begin{lemma}\label{lem:sym+qs=schurpositive} If a quasisymmetric function is symmetric and quasisymmetric Schur positive, then it is Schur positive.
\end{lemma}
\begin{proof} This follows since $s_\lambda = \sum _{\partitionof{\alpha} = \lambda}\qs _\alpha$.
\end{proof}

\begin{corollary}\label{cor:skewqsschurpositive} If a skew quasisymmetric Schur function is symmetric then it is Schur positive.
\end{corollary}
\begin{proof} This follows immediately from Theorem~\ref{BLvW-main}, which shows that every skew quasisymmetric Schur function is quasisymmetric Schur positive.
\end{proof}

The above corollary immediately begs the question of whether a combinatorial formula exists for the coefficients appearing in the Schur function expansion of a symmetric skew quasisymmetric Schur function. In answer, we will give explicit combinatorial interpretations for these, not only positive but also integral, coefficients in Theorem~\ref{the:NCLR} and Theorem~\ref{the:NCLRR}.

\begin{lemma}\label{lem:skewsasskewqs} Let $\beta$ be a composition such that $\partitionof{\beta}=\mu$, for some partition $\mu$. Let $\lambda$ be a partition. Then
$$s_{\lambda /\mu}=\sum _{\partitionof{\alpha}=\lambda}\qs _{\alpha \cskew \beta}.$$
\end{lemma}

\begin{proof}
This is a
consequence of \cite[Proposition 2.17]{BLvW}.
\end{proof}

\section{A left Littlewood-Richardson rule for symmetric skew quasisymmetric Schur functions}\label{sec:left}
Our aim in this section is to give a left Littlewood-Richardson rule for symmetric skew quasisymmetric Schur functions. We start by reminding the reader of a version of the Littlewood-Richardson rule before we state our own rule that resembles it. For this we need to define left Littlewood-Richardson reverse tableaux.

Given a skew shape $\lambda/\mu$, let $\LRRTL(\lambda/\mu)$ denote the set of all $\SSRT$s of shape $\lambda/\mu$ that have been constructed as follows. Given a partition $\nu = (\nu_1,\ldots,\nu_{\ell(\nu)})$, using the cover relations of $\Ly$ to place cells on the inner shape~$\mu$, place $\nu_{\ell(\nu)}$ cells containing $\ell(\nu)$ from left to right, $\nu _{\ell(\nu)-1} $ cells containing $\ell(\nu)-1$ from left to right, $\ldots$, $\nu_1$ cells containing $1$ from left to right such that the $k$-th $i$ from the \emph{left} is weakly \emph{left} of the $k$-th $i+1$ from the \emph{left}. Whenever the resulting $\SSRT$ is of shape~$\lambda/\mu$, this is an element in our set $\LRRTL(\lambda/\mu)$.
The elements of $\LRRTL(\lambda/\mu)$ are referred to as \emph{left Littlewood-Richardson reverse tableaux} of \emph{shape} $\lambda/\mu$.

Now recall the following classical Littlewood-Richardson rule for reverse tableaux (for example, see \cite[Theorem 2.4]{Kreiman}).
\begin{theorem}\label{the:LR}
Given $\lambda/\mu$ and a partition $\nu$, the coefficient of $s_{\nu}$ in $s_{\lambda/\mu}$ is equal to the cardinality of the set
\begin{align*}
\LRRTL_\nu(\lambda/\mu)=\{T\in \LRRTL(\lambda/\mu) \suchthat \cont (T)=\nu\}.
\end{align*}
\end{theorem}

An analogous rule holds for reverse composition tableaux, which we will state after establishing some more notation. Given a skew reverse composition shape $\alpha\cskew \beta$, let $\LRRCTL(\alpha\cskew\beta)$ denote the set of all $\SSRCT$s of shape $\alpha\cskew\beta$ that have been constructed as follows. Given a partition $\nu=(\nu_1,\ldots,\nu_{\ell(\nu)})$, using the cover relations of $\Lc$ to place cells on the inner shape~$\beta$, place $\nu_{\ell(\nu)}$ cells containing $\ell(\nu)$ from left to right, $\nu _{\ell(\nu)-1} $ cells containing $\ell(\nu)-1$ from left to right,$\ldots$, $\nu_1$ cells containing $1$ from left to right such that the $k$-th $i$ from the \emph{left} is weakly \emph{left} of the $k$-th $i+1$ from the \emph{left}.
Whenever the resulting $\SSRCT$ is of shape $\alpha\cskew\beta$, this is an element in our set $\LRRCTL(\alpha\cskew\beta)$. We will refer to the elements of $\LRRCTL(\alpha\cskew\beta)$ as \emph{left Littlewood-Richardson reverse composition tableaux} of \emph{shape} $\alpha\cskew\beta$.

Now we can state the left Littlewood-Richardson rule for symmetric skew quasisymmetric Schur functions; we will prove this at the end of the section. The reader should compare the statement in the following theorem with that of Theorem~\ref{the:LR}.

\begin{theorem}\label{the:NCLR} Given $\alpha \cskew \beta$ such that $\qs _{\alpha \cskew \beta}$ is symmetric, and a partition $\nu$, the coefficient of $s_{\nu}$ in $\qs_{\alpha\cskew \beta}$ is equal to the cardinality of the set
\begin{align*}
\LRRCTL_\nu(\alpha\cskew \beta)= \{\tau\in \LRRCTL(\alpha\cskew \beta)\suchthat \cont (\tau)=\nu\}.
\end{align*}
\end{theorem}

\begin{remark} \label{rem:classicalLR}
Notice that if $\alpha$ and $\beta$ are partitions of the same length then
$\alpha/\beta=\alpha\cskew \beta$ and we recover the classical Littlewood-Richardson rule, as the cover relations on $\Ly$ and $\Lc$ are identical since nothing is added to the leftmost column.
\end{remark}
The two examples that follow illustrate the expansion into Schur functions of two symmetric skew quasisymmetric Schur functions using Theorem~\ref{the:NCLR}.
\begin{example}\label{ex:newLR} We obtain
$$\qs _{(9,3,2,4)\cskew (4,1,1,3)} = s_{(7,1,1)} +s_{(7,2)}+ 2s_{(6,2,1)} +s_{(6,1,1,1)}+ s_{(5,2,1,1)} +s_{(5,2,2)}$$ from the
(complete list of) elements of $\LRRCTL((9,3,2,4)\cskew (4,1,1,3))$ below.
$$\tableau{\bullet&\bullet&\bullet&\bullet&1&1&1&1&1\\
\bullet&2&1\\
\bullet&1\\
\bullet&\bullet&\bullet&3}\quad
\tableau{\bullet&\bullet&\bullet&\bullet&1&1&1&1&1\\
\bullet&2&1\\
\bullet&1\\
\bullet&\bullet&\bullet&2}
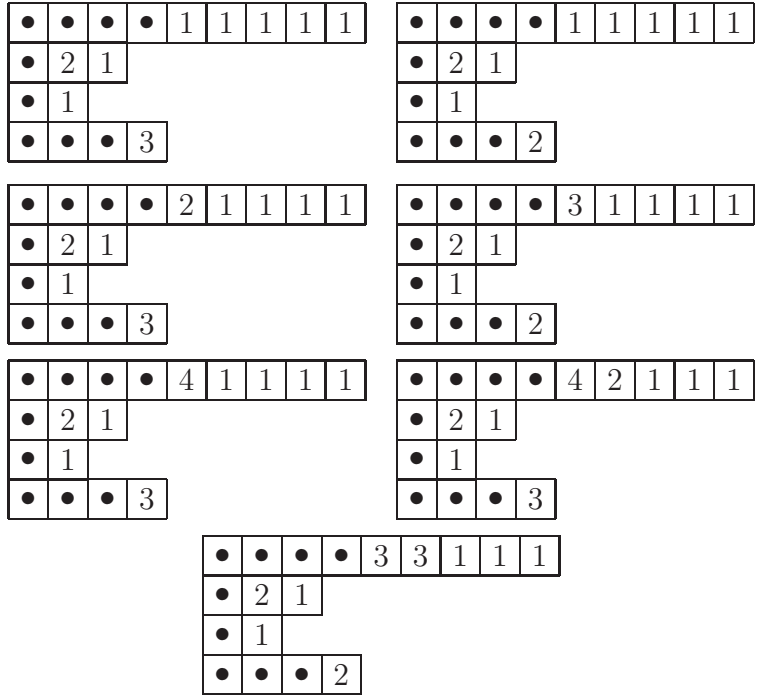
$$
$$
\tableau{\bullet&\bullet&\bullet&\bullet&2&1&1&1&1\\
\bullet&2&1\\
\bullet&1\\
\bullet&\bullet&\bullet&3}\quad
\tableau{\bullet&\bullet&\bullet&\bullet&3&1&1&1&1\\
\bullet&2&1\\
\bullet&1\\
\bullet&\bullet&\bullet&2}
$$
$$
\tableau{\bullet&\bullet&\bullet&\bullet&4&1&1&1&1\\
\bullet&2&1\\
\bullet&1\\
\bullet&\bullet&\bullet&3}\quad
\tableau{\bullet&\bullet&\bullet&\bullet&4&2&1&1&1\\
\bullet&2&1\\
\bullet&1\\
\bullet&\bullet&\bullet&3}
$$
$$
\tableau{\bullet&\bullet&\bullet&\bullet&3&3&1&1&1\\
\bullet&2&1\\
\bullet&1\\
\bullet&\bullet&\bullet&2}
$$
\end{example}

\begin{example}\label{ex:newLR2} We obtain $$\qs _{(2,3,3,2)\cskew (1,2,1)} = s_{(2,2,1,1)} +s_{(2,2,2)}$$ from the (complete list of)
elements of $\LRRCTL((2,3,3,2)\cskew (1,2,1))$ below.
$$\tableau{
1&1\\
\bullet&3&2\\
\bullet&\bullet& 4\\
\bullet& 2}\qquad
\tableau{
1&1\\
\bullet&3&3\\
\bullet&\bullet& 2\\
\bullet& 2}$$
\end{example}
We will spend the remainder of this section establishing Theorem~\ref{the:NCLR}. Henceforth in this section, fix compositions $\alpha$ and $\beta$ such that $\beta <_{c} \alpha$. Suppose further that $\partitionof{\alpha}=\lambda, \partitionof{\beta}=\mu$.
Then observe that the bijection $\rho^{-1}_\beta$ restricts to a bijection
\begin{align*}
\rhoc ^{-1} _\beta :
\LRRTL(\lambda /\mu) \to  \bigcup _{\partitionof{\gamma }= \lambda}\LRRCTL(\gamma \cskew \beta )
\end{align*}
and hence we have a set partition
\begin{align*}\label{eq:unions}
\rhoc ^{-1} _\beta(\LRRTL(\lambda /\mu)) =
\LRRCTL(\alpha \cskew \beta ) \cup \bigcup _{\partitionof{\gamma }= \lambda \atop \gamma \neq \alpha}\LRRCTL(\gamma \cskew \beta ).
\end{align*}
We will establish that only the elements of $\LRRCTL(\alpha\cskew\beta)$ contribute to the quasisymmetric Schur function expansion of $\qs_{\alpha\cskew\beta}$ in a sense that will become precise once we prove the following lemma.

\begin{lemma}\label{lem:rect}Let $\nu$ be a partition and $\rtau \in \LRRCTL_\nu(\alpha \cskew \beta)$. Then $\rect(\stan(\rtau))=\rtau _\nu$.
\end{lemma}

\begin{proof} Let $\nu=(\nu_1,\ldots,\nu_{\ell(\nu)})$. We will prove this lemma by establishing, via induction on the number of insertions in the rectification process, that when $j$ is inserted it is placed in the leftmost column if $j\in \{ \nu_1, \nu_1+\nu_2, \ldots\}$ or placed to the immediate right of $j+1$ otherwise.

Note that the first number we will insert is $\nu _1$, which will be placed in the leftmost column. Now assume the result is true for $m$ insertions. Let the $m+1$-th insertion be of some number~$j$. Then there are two cases to consider.

If $j\in \{\nu _1, \nu_1+\nu_2, \ldots \}$, then since $\rtau \in \LRRCTL(\alpha \cskew \beta)$ this implies that the pre-image under standardization of $j$ is the leftmost occurrence of some number $i$, and hence all numbers inserted before $j$ are smaller than $j$. Thus the insertion algorithm will place $j$ in the leftmost column.

If $j\not\in \{\nu _1, \nu_1+\nu_2, \ldots \}$, then since $\rtau \in \LRRCTL(\alpha \cskew \beta)$ this implies that the pre-image under standardization of $j$ is the $k$-th occurrence from the left of some number $i$, where $k>1$, which is weakly left of the $k$-th $i+1$ from the left, which in turn is weakly left of the $k$-th $i+2$ from the left, etc. Thus by our induction hypothesis no row whose rightmost cell contains a number larger than $j$ is longer than the row containing $j+1$, \new{where $j+1$} has already been inserted since its pre-image is the
$(k-1)$-th occurrence of $i$ from the left. Thus the insertion algorithm will place $j$ to the immediate right of $j+1$ as desired.
\end{proof}

Consider now a partition $\nu\vdash \vert\alpha\cskew\beta\vert$. Recall Theorem~\ref{BLvW-main}, which implies that the coefficient of $\qs_\nu$ in $\qs_{\alpha\cskew\beta}$ equals the number of $\SRCT$s of shape $\alpha\cskew\beta$ that rectify to $\tau_{\nu}$. From Lemma \ref{lem:rect}, we obtain the following inclusion.
\begin{align*}
\{\stan (\tau)\suchthat \tau\in\LRRCTL_\nu(\alpha\cskew\beta)\} \subseteq \{\tau \suchthat \tau\in \SRCT (\alpha\cskew\beta)\text{ and }\rect(\tau)=\tau_{\nu}\}
\end{align*}
The next proposition establishes that the inclusion above is, in fact, an equality of sets.
\begin{proposition}\label{prop:LRcoeff} Given a partition $\nu$, the coefficient
 of $\qs _\nu$ in $\qs _{\alpha \cskew \beta}$ is
\begin{align*}
C^\alpha_{\nu \beta}=|\LRRCTL_\nu(\alpha \cskew \beta)|.
\end{align*}
\end{proposition}

\begin{proof} From Lemma~\ref{lem:rect} and Theorem~\ref{BLvW-main} we have that each $\rtau \in \LRRCTL_\nu(\alpha \cskew \beta)$ contributes one towards the coefficient of $\qs _\nu$ in $\qs _{\alpha\cskew \beta}$, and similarly each $\rtau \in \LRRCTL_\nu(\gamma \cskew \beta)$ contributes one towards the coefficient of $\qs _\nu$ in $\qs _{\gamma\cskew \beta}$ for $\partitionof{\gamma} = \partitionof{\alpha}$. Classically, by Theorem~\ref{the:LR} we know that the number of $\rT \in \LRRTL_{\nu}(\lambda / \mu)$
is the coefficient of $s _\nu$ in $s _{\lambda/ \mu}$.

Furthermore by Lemma~\ref{lem:skewsasskewqs}, since $\partitionof{\alpha}=\lambda, \partitionof{\beta}=\mu$, we have that
\begin{equation}\label{eq:skewasskewqs} s_{\lambda /\mu}=\sum _{\partitionof{\gamma}=\lambda}\qs _{\gamma \cskew \beta}\end{equation}and we already know that
$$\rhoc ^{-1} _\beta (\LRRTL(\lambda /\mu)) = \LRRCTL(\alpha \cskew \beta ) \cup \bigcup _{\partitionof{\gamma }= \lambda \atop \gamma \neq \alpha}\LRRCTL(\gamma \cskew \beta ) $$and
$$s_\nu = \qs_\nu + \sum _{\partitionof{\gamma} = \nu \atop \gamma \neq \nu}\qs _\gamma.$$Taking the coefficient of $\qs _\nu$ on each side of Equation~\eqref{eq:skewasskewqs} yields the result.
\end{proof}
The proposition above immediately implies the following corollary, after which we give a proof for Theorem~\ref{the:NCLR}.
\begin{corollary}\label{cor:stbij} Given a partition $\nu$, standardization restricts to a bijection
$$\stan : \LRRCTL_\nu (\alpha \cskew \beta)
\to \{\rtau \in \SRCT(\alpha \cskew \beta) \mid \rect(\rtau)=\rtau _\nu\} \:.$$
\end{corollary}

\begin{proof}(of Theorem~\ref{the:NCLR}) If $\qs _{\alpha\cskew\beta}$
is symmetric, then
the coefficient of $s_\nu$ in $\qs _{\alpha\cskew\beta}$
is the coefficient of $\qs_\nu$ in $\qs _{\alpha\cskew\beta}$
since $s_\nu = \qs_\nu + \sum _{\partitionof{\gamma} =
\nu \atop \gamma \neq \nu}\qs _\gamma$, and the
supports of the Schur functions in the basis of quasisymmetric Schur functions
are disjoint.
The result now follows by Proposition~\ref{prop:LRcoeff}.
\end{proof}

\section{A right Littlewood-Richardson rule for symmetric skew quasisymmetric Schur functions}
In this section, we will give a right Littlewood-Richardson rule for symmetric skew quasisymmetric Schur functions. The classical version that our rule will resemble is different from the one stated in the previous section. We will first define right Littlewood-Richardson reverse tableaux.

Given a skew shape $\lambda/\mu$, let $\LRRTR(\lambda/\mu)$ denote the set of all $\SSRT$s of shape $\lambda/\mu$ that have been constructed as follows. Given a partition $\nu=(\nu_1,\ldots,\nu_{\ell(\nu)})$, using the cover relations of $\Ly$ to place cells on the inner shape~$\mu$, place $\nu _1$ cells containing $\ell(\nu) $ from left to right, $\nu _2$ cells containing $\ell(\nu)-1 $ from left to right,$\ldots$, $\nu_{\ell(\nu)}$ cells containing $1$ from left to right, such that the $k$-th $i+1$ from the \emph{right} is weakly \emph{right} of the $k$-th $i$ from the \emph{right}. Whenever the resulting $\SSRT$ is of shape~$\lambda/\mu$, this is an element in our set $\LRRTR(\lambda/\mu)$.
Note that the aforementioned procedure implies that once such an $\SSRT$ has been constructed, as we read the entries of each column taken from right to left (where within a column we read the entries from largest to smallest) the number of $i+1$s we have read is always weakly greater than the number of $i$s we have read. Note the converse also holds. The elements of $\LRRTR(\lambda/\mu)$ will be referred to as \emph{right Littlewood-Richardson reverse tableaux} of \emph{shape}~$\lambda/\mu$.

Recall now the following classical Littlewood-Richardson rule for reverse tableaux
(for example, see \cite[Section 5.2]{fulton-1}).
\begin{theorem}\label{the:LRR} Given $\lambda / \mu$ and a partition $\nu$, the coefficient of $s_{\nu}$ in $s_{\lambda/\mu}$ is equal to the cardinality of the set
\begin{align*}
\LRRTR _{\reverse{\nu}}(\lambda/\mu) =
\{T\in \LRRTR(\lambda/\mu)\suchthat \cont(T)=\reverse{\nu}\}.
\end{align*}
\end{theorem}

An analogous rule holds for reverse composition tableaux, and to state it we need some notation. Given a skew reverse composition shape $\alpha\cskew \beta$, let $\LRRCTR(\alpha\cskew\beta)$ denote the set of all $\SSRCT$s of shape $\alpha\cskew\beta$ that have been constructed as follows. Given a partition $\nu=(\nu_1,\ldots,\nu_{\ell(\nu)})$, using the cover relations of $\Lc$ to place cells on the inner shape~$\beta$, place $\nu _1$ cells containing $\ell(\nu) $ from left to right, $\nu _2$ cells containing $\ell(\nu)-1 $ from left to right,$\ldots$, $\nu_{\ell(\nu)}$ cells containing $1$ from left to right, such that the $k$-th $i+1$ from the \emph{right} is weakly \emph{right} of the $k$-th $i$ from the \emph{right}.
Whenever the resulting $\SSRCT$ is of shape $\alpha\cskew\beta$, this is an element in our set $\LRRCTR(\alpha\cskew\beta)$.
Again, we have that once such an $\SSRCT$ has been constructed, as we read the entries of each column taken from right to left (where within a column we read the entries from largest to smallest) the number of $i+1$s we have read is always weakly greater than the number of $i$s we have read. Note that the converse also holds. We will refer to the elements of $\LRRCTR(\alpha\cskew \beta)$ as \emph{right Littlewood-Richardson reverse composition tableaux} of \emph{shape} $\alpha\cskew\beta$. Armed with this notation, we are ready to state our analogue of Theorem~\ref{the:LRR} for symmetric skew quasisymmetric Schur functions.

\begin{theorem}\label{the:NCLRR}
Given $\alpha \cskew \beta$ such that $\qs _{\alpha \cskew \beta}$ is symmetric, and a partition $\nu$, the coefficient of $s_{\nu}$ in $\qs_{\alpha\cskew \beta}$ is equal to the cardinality of the set
\begin{align*}
\LRRCTR _{\reverse{\nu}}(\alpha\cskew \beta) =
\{\tau \in \LRRCTR(\alpha\cskew\beta)\suchthat \cont(\tau)=\reverse{\nu}\}.
\end{align*}
\end{theorem}

\begin{remark} \label{rem:classicalLRR}
As before, notice that if $\alpha$ and $\beta$ are partitions of the same length then we recover the classical Littlewood-Richardson rule, as the cover relations on $\Ly$ and $\Lc$ are identical since nothing is added to the leftmost column.
\end{remark}
Next we consider the same examples as we did in the previous section, but use Theorem~\ref{the:NCLRR} instead.
\begin{example}\label{ex:new:LRR} We obtain $$\qs _{(9,3,2,4)\cskew (4,1,1,3)} = s_{(7,1,1)} +s_{(7,2)}+  2s_{(6,2,1)} +s_{(6,1,1,1)}+ s_{(5,2,1,1)} +s_{(5,2,2)}$$ from the (complete list of) elements of $\LRRCTR ((9,3,2,4)\cskew(4,1,1,3))$ below.
$$\tableau{\bullet&\bullet&\bullet&\bullet&3&3&3&3&3\\
\bullet&3&2\\
\bullet&1\\
\bullet&\bullet&\bullet&3}\quad
\tableau{\bullet&\bullet&\bullet&\bullet&2&2&2&2&2\\
\bullet&2&1\\
\bullet&1\\
\bullet&\bullet&\bullet&2}
$$
$$
\tableau{\bullet&\bullet&\bullet&\bullet&3&3&3&3&3\\
\bullet&2&2\\
\bullet&1\\
\bullet&\bullet&\bullet&3}\quad
\tableau{\bullet&\bullet&\bullet&\bullet&3&3&3&3&3\\
\bullet&3&1\\
\bullet&2\\
\bullet&\bullet&\bullet&2}
$$
$$
\tableau{\bullet&\bullet&\bullet&\bullet&4&4&4&4&4\\
\bullet&4&2\\
\bullet&1\\
\bullet&\bullet&\bullet&3}\quad
\tableau{\bullet&\bullet&\bullet&\bullet&4&4&4&4&4\\
\bullet&3&2\\
\bullet&1\\
\bullet&\bullet&\bullet&3}
$$
$$
\tableau{\bullet&\bullet&\bullet&\bullet&3&3&3&3&3\\
\bullet&2&1\\
\bullet&1\\
\bullet&\bullet&\bullet&2}
$$
\end{example}

\begin{example}\label{ex:newLRR2} We obtain $$\qs _{(2,3,3,2)\cskew (1,2,1)} = s_{(2,2,1,1)} +s_{(2,2,2)}$$ from the (complete list of) elements of $\LRRCTR ((2,3,3,2)\cskew(1,2,1))$ below.
$$\tableau{
3&2\\
\bullet&4&4\\
\bullet&\bullet& 3\\
\bullet& 1}\qquad
\tableau{
1&1\\
\bullet&3&3\\
\bullet&\bullet& 2\\
\bullet& 2}$$
\end{example}

For the remainder of this section, fix compositions $\alpha$ and $\beta$ such that $\beta <_{c} \alpha$. Suppose further that $\partitionof{\alpha}=\lambda, \partitionof{\beta}=\mu$. Similar to the previous section, we observe that we have the following set partition.
\begin{equation*}\label{eq:unionsr}
\rhoc ^{-1} _\beta (\LRRTR(\lambda /\mu)) = \LRRCTR(\alpha \cskew \beta ) \cup \bigcup _{\partitionof{\gamma }= \lambda \atop \gamma \neq \alpha}\LRRCTR(\gamma \cskew \beta )
\end{equation*}
As in the previous section, we will establish that only the elements of $\LRRCTR(\alpha \cskew \beta )$ contribute to the quasisymmetric Schur function expansion of $\qs_{\alpha\cskew\beta}$. Towards this end, we need the following lemma.

\begin{lemma}\label{lem:rectr}
Let $\nu$ be a partition and $\rtau \in \LRRCTR_{\reverse{\nu}}(\alpha \cskew \beta)$. Then $\rect(\stan(\rtau)) =\rtau _{\reverse{\nu}}$.
\end{lemma}

\begin{proof}
Let $\rT \in \LRRTR _{\reverse{\nu}}(\lambda/\mu)$. It follows from \cite[Proposition 2.3]{HLMvW11a} and \cite[Theorem 3.4]{BLvW} that
$$\rect(\stan(\rT))=\rT _{\nu}$$
where we recall that $\rT _{\nu}$ is the canonical reverse tableau
of shape~$\nu$. Now \cite[Proposition 3.1]{mason-1} implies that for any $T\in \SRT (\lambda/\mu)$, we have that $$\rect(\rho_{\beta}^{-1}(T))= \rho_{\varnothing}^{-1}(\rect (T)).$$ Note also that $\rhoc_{\varnothing}^{-1}(\rT_{\nu})=\rtau_{\reverse{\nu}}$.
Let $\rT _{\nu} (\lambda/\mu) $ denote the set of all $\SRT$s of shape $\lambda /\mu$ that rectify to $\rT _{\nu}$, and let $\rtau _{\reverse{\nu}}(\gamma \cskew \beta)$ denote the set of all $\SRCT$s of shape $\gamma \cskew \beta$ that rectify to $\rtau _{\reverse{\nu}}$.
Then we have the following commutative diagram that proves the claim.
$$\xymatrix{
\LRRTR _{\reverse{\nu}}(\lambda / \mu)   \ar@{->}[d] _{\rho ^{-1} _\beta}  \ar@{->}[rr]^{\stan} && \rT _{\nu}(\lambda/\mu) \ar@{->}[d] ^{\rho ^{-1} _\beta} \\
\displaystyle\bigcup _{\partitionof{\gamma}=\lambda}\LRRCTR _{\reverse{\nu}}(\gamma\cskew \beta)  \ar@{->}[rr]_{\stan}  &&
\displaystyle\bigcup _{\partitionof{\gamma}=\lambda}
\rtau _{\reverse{\nu}}(\gamma \cskew \beta)}
$$
\end{proof}

Our strategy now is similar to the one we employed in the previous section. Consider a partition $\nu$ of $\vert\alpha\cskew\beta\vert$. By Theorem~\ref{BLvW-main}, we have that the coefficient of $\qs_{\reverse{\nu}}$ in $\qs_{\alpha\cskew\beta}$ is the number of $\SRCT$s of shape $\alpha\cskew\beta$ that rectify to $\tau_{\reverse{\nu}}$. From Lemma \ref{lem:rectr}, we obtain the following inclusion.
\begin{align*}
\{\stan (\tau)\suchthat \tau\in \LRRCTR_{\reverse{\nu}}(\alpha\cskew\beta)\} \subseteq \{\tau \suchthat \tau\in \SRCT (\alpha\cskew\beta)\text{ and }\rect(\tau)=\tau_{\reverse{\nu}}\}
\end{align*}

The next proposition establishes that the inclusion above is an equality of sets.
\begin{proposition}\label{prop:LRRcoeff} Given a partition $\nu$, the coefficient
of $\qs _{\reverse{\nu}}$ in $\qs _{\alpha \cskew \beta}$
is
\begin{align*}
C^\alpha_{\nu^r \beta}=
|\LRRCTR _{\reverse{\nu}}(\alpha\cskew \beta)|
.
\end{align*}
\end{proposition}

\begin{proof} From Lemma~\ref{lem:rectr} and Theorem~\ref{BLvW-main} we have that each $\rtau \in \LRRCTR_{\reverse{\nu}}(\alpha \cskew \beta)$ contributes one towards the coefficient of $\qs _{\reverse{\nu}}$ in $\qs _{\alpha\cskew \beta}$, and similarly each $\rtau \in \LRRCTR_{\reverse{\nu}}(\gamma \cskew \beta)$
contributes one towards the coefficient of $\qs _{\reverse{\nu}}$ in $\qs _{\gamma\cskew \beta}$ for $\partitionof{\gamma} = \partitionof{\alpha}$. Classically by Theorem~\ref{the:LRR} we know that the number of $\rT \in \LRRTR_{\reverse{\nu}}(\lambda / \mu)$
is the coefficient of $s _\nu$ in $s _{\lambda/ \mu}$.

Furthermore by Lemma~\ref{lem:skewsasskewqs}, since $\partitionof{\alpha}=\lambda, \partitionof{\beta}=\mu$, we have that
\begin{equation}\label{eq:skewasskewqsr} s_{\lambda /\mu}=\sum _{\partitionof{\gamma}=\lambda}\qs _{\gamma \cskew \beta}\end{equation}and we already know that
$$\rhoc ^{-1} _\beta (\LRRTR(\lambda /\mu)) = \LRRCTR(\alpha \cskew \beta ) \cup \bigcup _{\partitionof{\gamma }= \lambda \atop \gamma \neq \alpha}\LRRCTR(\gamma \cskew \beta ) $$and
$$s_\nu = \qs_{\reverse{\nu}} + \sum _{\partitionof{\gamma} = \nu \atop \gamma \neq \reverse{\nu}}\qs _\gamma.$$Taking the coefficient of $\qs _{\reverse{\nu}}$ on each side of Equation~\eqref{eq:skewasskewqsr} yields the result.
\end{proof}
This gives us the following corollary.
\begin{corollary}\label{cor:stbijr} Given a partition $\nu$, standardization
restricts to a bijection
$$\stan : \LRRCTR_{\reverse{\nu}} (\alpha \cskew \beta)
\to \{\rtau \in \SRCT(\alpha \cskew \beta) \mid \rect(\rtau)=\rtau _{\reverse{\nu}}\} \:.$$
\end{corollary}

Now we can give a proof of Theorem~\ref{the:NCLRR}.
\begin{proof}(of Theorem~\ref{the:NCLRR})
If $\qs _{\alpha\cskew\beta}$ is symmetric, then
the coefficient of $s_\nu$ in $\qs _{\alpha\cskew\beta}$
is the coefficient of $\qs_{\reverse{\nu}}$ in $\qs _{\alpha\cskew\beta}$
since $s_\nu = \qs_{\reverse{\nu}} + \sum _{\partitionof{\gamma} =
\nu \atop \gamma \neq \reverse{\nu}}\qs _\gamma$, and the
supports of the Schur functions in the basis of quasisymmetric Schur functions
are disjoint. The result now follows by Proposition~\ref{prop:LRRcoeff}.
\end{proof}

\section{The classification of symmetric skew quasisymmetric Schur functions}\label{sec:uniform}
The aim of this section is to use the left and the right Littlewood-Richardson
rules obtained earlier to combinatorially classify symmetric skew quasisymmetric
Schur functions. To this end, we need three lemmas and a proposition.
For the remainder of the section
fix compositions $\alpha$ and $\beta$ where $\beta <_{c} \alpha$,
$\ell(\alpha)=k+\ell$ and $\ell(\beta)=k$.
Our next lemma is a result about the shape of the tableaux obtained during the process of rectification of an SRCT.
\begin{lemma}\label{lem:insertionlength} When rectifying $\rtau \in \SRCT(\alpha \cskew \beta)$, after the insertion of the $j$-th column, to form $\rtau ^j$, no row of $\rtau ^j$ contains more than $j$ cells.
\end{lemma}

\begin{proof} We proceed by induction on the number of columns. The result is trivially true for the first column. Now assume that the result is true for $j-1$ columns. On the insertion of the $j$-th column into $\rtau ^{j-1}$ observe that since we insert entries from smallest to largest, if $i\rightarrow \rtau _{j-1}$ results in $i$ being placed in column $j$, the subsequent insertions of larger numbers cannot result in a number being placed in column $j+1$ since row entries must weakly decrease when read from left to right. So no row of $\rtau ^j$ contains more than $j$ cells.
\end{proof}

We will now use the lemma above to prove that $\qs_{\alpha\cskew\beta}$ contains a distinguished summand in its quasisymmetric Schur function expansion. Not only is this summand easy to compute, it is also crucial for our desired classification.
\begin{proposition}\label{prop:rowfilling}
Let $\alpha \cskew \beta$ have $m$ nonempty rows. Let $\delta _i$ be the number of cells in the $i$-th nonempty row from the top in $\alpha\cskew\beta$, for $1\leq i \leq m$, and set $\delta = (\delta _1, \ldots ,\delta _m)\vDash|\alpha\cskew\beta|$. Then $\qs _{\alpha \cskew \beta}$ always contains the summand $\qs _\delta$.
\end{proposition}

\begin{proof}
Observe that the proposition will follow by Theorem~\ref{BLvW-main} if we can show that there exists a $\rtau \in\SRCT(\alpha \cskew \beta)$ whose rectification is $\rtau _\delta$. Let the upper shape of $\alpha \cskew \beta$ be denoted by $\alpha \cskew \beta \mid _u$ and the lower shape of $\alpha \cskew \beta$ be denoted by $\alpha \cskew \beta \mid _l$. Furthermore, let $\delta ^u = (\delta _1, \ldots ,\delta _\ell)$ and $\delta ^l = (\delta _{\ell+1}, \ldots ,\delta _m)$.

Let $r_1, \ldots , r_t$\new{,} where $t= m-\ell$, be the rows with a nonzero number of cells in $\alpha \cskew \beta \mid _l$ that we will fill in the following construction. We will now order the rows according to how far to the right their leftmost empty cell lies. In particular, we say $r_i>r_j$ if either the leftmost empty cell in row $r_i$ lies to the right of the the leftmost empty cell in row $r_j$, or if the leftmost empty cell in both rows are in the same column but $i<j$. Thus we have a linear order on the rows with a nonzero number of cells in $\alpha\cskew \beta\mid_l$, say,
$$r_{i_t}>r_{i_{t-1}}> \cdots > r_{i_1}.$$We will now fill the cells of $\alpha \cskew \beta \mid _l$ to create $\rtau ' $ as follows.
\begin{enumerate}
\item Place the integers from $t$ down to $1$ in the leftmost empty cells of rows taken in the order $r_{i_t}$ down to $r_{i_1}$.
\item Repeat until all the cells are filled.
\end{enumerate}

By construction $\rtau ' \in \LRRCTL (\alpha \cskew \beta \mid _l)$ once we confirm that $\rtau '$ is an $\SSRCT$. First note that after every iteration, the linear order on the rows is preserved, but the chain may decrease in length. Hence in every application of the first step, the entries in the rows constructed will weakly decrease when read from left to right. By construction there are no entries in the first column, thus we only need confirm that the triple condition always holds once all the cells are filled.

Consider rows $r_i$ and $r_j$ where $r_i$ is a row higher in $\alpha \cskew \beta \mid _l$ than $r_j$. Let $r_j$ have a cell in column $c$ containing entry $\rtau ' (r_j, c)$. There are two cases to consider. If $r_i > r_j$ then the leftmost cell of $r_i$ in the outer shape is weakly right of the leftmost cell of $r_j$ in the outer shape. Thus by the cover relations in $\Lc$ if $r_i$ has a cell in column $c-1$, then $r_i$ has a cell in column $c$. Let their respective entries be $\rtau ' (r_i, c-1) $ and $\rtau ' (r_i, c) $.
By construction we have $\rtau ' (r_j, c)<\rtau ' (r_i, c)$. If $r_i < r_j$ then the leftmost cell of $r_i$ in the outer shape is strictly left of the leftmost cell of $r_j$ in the outer shape, and hence by construction we have $\rtau ' (r_j, c)>\rtau ' (r_i, c-1)$, if there exists a cell in row $r_i$ and column $c-1$ with entry $\rtau ' (r_i, c-1)$. Thus $\rtau '$ is an $\SSRCT$ and hence $\rtau ' \in \LRRCTL (\alpha \cskew \beta \mid _l)$. So by \cite[Corollary 5.2]{BLvW} and Theorem~\ref{the:NCLR}, $\qs _{\partitionof{\delta ^l}}$ and hence $\qs _{\delta ^l}$ is a summand of $\qs _{\alpha \cskew \beta \mid _l}$. In particular, by Theorem~\ref{BLvW-main} this implies that there exists a $\rtau '' \in \SRCT(\alpha \cskew \beta \mid _l)$ that rectifies to $\rtau _{\delta ^l}$.

Now consider $\rtau\in \SRCT(\alpha \cskew \beta)$ whose upper shape $\alpha \cskew \beta \mid _u$ consisting of the first $\ell$ rows is filled as the canonical reverse composition tableau of shape $\delta^u$, and the remaining nonempty rows are those of $\rtau '' $ with $\vert \delta^u\vert$ added to each entry.
We claim that the rectification of $\rtau$ is $\rtau _\delta$.

Observe the entries in the top $\ell$ rows trivially rectify to $\rtau _{\delta ^u}$, since by our insertion procedure and Lemma~\ref{lem:insertionlength} the leftmost column when inserted will be $\delta ^u _1,  \delta ^u _1 + \delta ^u _2, \ldots $ and, thereafter, if $i$ is in column $j$ of $\rtau$ to the immediate right of $i+1$ then $i$ will be placed in column $j$ to the immediate right of $i+1$ during rectification. Since every entry in the remaining rows of $\rtau$ is larger than those in the first $\ell$ rows, they can never be placed in the same row as any number appearing in the first $\ell$ rows of $\rtau$ during rectification. Hence since $\rtau ''$ rectifies to $\rtau _{\delta ^l}$, the remaining rows will rectify to $\rtau _{\delta ^l}$ with $\vert \delta^u \vert$ added to each entry.

Hence $\rtau $ rectifies to an $\SRCT$ whose top $\ell$ rows are $\rtau _{\delta ^u }$ and whose remaining rows are $\rtau _{\delta ^l}$ with $\vert \delta^u\vert$ added to each entry, that is, $\rtau _\delta$.\end{proof}

Our next lemma shows that Theorem~\ref{the:NCLR} places a constraint on what the upper shape of $\alpha\cskew\beta$ can be.
\begin{lemma}\label{lem:sympart} If $\qs _{\alpha \cskew \beta}$ is symmetric then the upper shape of $\alpha \cskew \beta$ is a partition, that is, $\alpha _1 \geq \cdots \geq \alpha _\ell$.
\end{lemma}

\begin{proof} Assume, for the sake of contradiction, that $\qs _{\alpha \cskew \beta}$ is symmetric but the upper shape of $\alpha \cskew \beta$ is not a partition.

We will now try to construct $\rtau \in\LRRCTL_\nu (\alpha \cskew \beta)$ for some partition $\nu$. \new{By definition, the leftmost $i$ has to be weakly left of the leftmost $i+1$. Thus,} if we have such a $\rtau$, then by construction we know the leftmost column read from top to bottom must read $1, \ldots, \ell$. Then the second left must contain entries $1, \ldots ,\ell _2$ for some $\ell _2 \leq \ell$, the third left must contain entries $1, \ldots ,\ell _3$ for some $\ell _3 \leq \ell _2$ and so on. By the definition of $\SSRCT$s it follows that the entries in row $i$ for $1\leq i \leq \ell$ will be equal to $i$. Furthermore, since the upper shape of $\alpha \cskew \beta$ is not a partition, there exists some $\alpha _j >\alpha _{j-1}$ ($j$ minimal) for $2\leq j\leq \ell$. Note an $(\alpha _{j-1}+1)$-th $j-1$ from the left cannot appear weakly left of the $(\alpha _{j-1}+1)$-th $j$ from the left, as the definition of $\SSRCT$s guarantees that this $j-1$ cannot belong to the lower shape in column $\alpha _{j-1}+1$.

Thus, if the upper shape of $\alpha \cskew \beta$ is not a partition then it is not possible to construct $\rtau \in \LRRCTL_\nu (\alpha \cskew \beta)$ for some partition $\nu$. Hence by Theorem~\ref{the:NCLR} the coefficient of $s_\nu$ in $\qs _{\alpha \cskew \beta}$ is zero for all partitions $\nu$, and hence $\qs _{\alpha \cskew \beta}=0$. However, by Proposition~\ref{prop:rowfilling} $\qs _{\alpha \cskew \beta}\neq0$, a contradiction. \end{proof}
On the other hand, as the following lemma shows, using Theorem~\ref{the:NCLRR} we get a different constraint on what the upper shape of $\alpha\cskew \beta$ can be.
\begin{lemma}\label{lem:symrevpart} If $\qs _{\alpha \cskew \beta}$ is symmetric then the upper shape of $\alpha \cskew \beta$ is the reverse of a partition, that is, $\alpha _1 \leq \cdots \leq \alpha _\ell$.
\end{lemma}

\begin{proof} Assume, for the sake of contradiction, that $\qs _{\alpha \cskew \beta}$ is symmetric but the upper shape of $\alpha \cskew \beta$ is not the reverse of a partition. By Proposition~\ref{prop:rowfilling}, if $\alpha \cskew \beta$ has $m$ nonempty rows then $\qs _{\alpha \cskew \beta}$ contains the summand $\qs _\delta$ where $\delta = (\delta _1, \ldots, \delta _m) \vDash |\alpha \cskew \beta|$, and $\delta _i$ for $1\leq i \leq m$ is the number of cells in the $i$-th nonempty row from the top in $\alpha \cskew \beta$. Note that $\delta$ is not the reverse of a partition since the upper shape of $\alpha \cskew \beta$ is not the reverse of a partition. Since $\qs _{\alpha \cskew \beta}$ is symmetric it follows that the coefficient of $\qs _{\reverse{\partitionof{\delta}}}$ in $\qs _{\alpha \cskew \beta}$ is nonzero.

We will now try to construct $\rtau \in \LRRCTR (\alpha \cskew \beta)$ with content $\reverse{\partitionof{\delta}}$. Consider the transpose of $\partitionof{\delta}$, $\varepsilon = (\varepsilon _1, \ldots, \varepsilon _{\partitionof{\delta}_1})$, where $\varepsilon _1 = m$. Then if such a $\rtau$ exists, as the entries in the rows weakly decrease from left to right, and the $k$-th $i+1$ from the right appears weakly right of the $k$-th $i$ from the right it follows that as we traverse the columns of $\alpha \cskew \beta$ from right to left reading the entries at the end of each row, we read $m, m-1, \ldots, 3,2,1$ (where if more than one row ends in the same column we read the relevant entries from largest to smallest). By the same argument, it follows that as we traverse the columns of $\alpha \cskew \beta$ from right to left reading the $k$-th entries from the end of each row, we read $m, m-1, \ldots, m+1-\varepsilon_k$.

Since the upper shape of $\alpha \cskew \beta$ is not the reverse of a partition, there exists some $\alpha _j<\alpha _{j-1}$ ($j$ minimal) for $2\leq j\leq \ell$ and as we traverse the columns of $\alpha \cskew \beta$ from right to left reading the $\alpha _j$-th entries from the end of each row, we observe that the leftmost cell in row $j$ is $m+1-\varepsilon _{\alpha _j}$. We also observe that the $\alpha _j$-th entry from the end of row $j-1$ is strictly greater than $m+1-\varepsilon _{\alpha _j}$. Thus since row entries weakly decrease from left to right the leftmost cell in row $j-1$ is strictly greater than $m+1-\varepsilon _{\alpha _j}$.

Hence no such $\rtau$ exists since entries in the leftmost column must increase when read from top to bottom by the definition of an $\SSRCT$. Thus, by Theorem~\ref{the:NCLRR}, the coefficient of $\qs _{\reverse{\partitionof{\delta}}}$ in $\qs _{\alpha \cskew \beta}$ is zero, a contradiction. \end{proof}

Now, using Lemmas \ref{lem:sympart} and \ref{lem:symrevpart} in conjunction readily yields the following theorem, which combinatorially classifies symmetric skew quasisymmetric Schur functions.
\begin{theorem}\label{the:uniform} $\qs _{\alpha \cskew \beta}$ is symmetric if and only if $\alpha \cskew \beta$ is uniform.
\end{theorem}

\begin{proof}
If $\alpha \cskew \beta$ is uniform then $\qs _{\alpha \cskew \beta}$ is symmetric \cite[Corollary 5.2]{BLvW}. Conversely, let $\qs _{\alpha \cskew \beta}$ be symmetric. Then by Lemma~\ref{lem:sympart} the upper shape of $\alpha \cskew \beta$ is a partition, that is, $\alpha _1 \geq \cdots \geq \alpha _\ell$. Meanwhile by Lemma~\ref{lem:symrevpart} the upper shape of $\alpha \cskew \beta$ is the reverse of a partition, that is, $\alpha _1 \leq \cdots \leq \alpha _\ell$. Hence
$$\alpha _1 = \cdots = \alpha _\ell,$$that is, $\alpha \cskew \beta$ is uniform.
\end{proof}
We also obtain the following corollary about \new{one place} where the symmetry is broken in nonsymmetric skew quasisymmetric Schur functions, with which we conclude.
\begin{corollary}\label{cor:nonsym} Let $\qs _{\alpha \cskew \beta}$ be a nonsymmetric skew quasisymmetric Schur function, $\alpha \cskew \beta$ have $m$ nonempty rows and $\delta = (\delta _1, \ldots ,\delta _m)$ be the composition of $|\alpha \cskew \beta|$ where $\delta _i$ is the number of cells in the $i$-th nonempty row from the top in $\alpha \cskew \beta$, for $1\leq i \leq m$. Then\cnew{, 
 while $\qs _\delta$ appears,} not all terms $\qs _\gamma$ where $\tilde{\delta}=\tilde{\gamma}$ appear in the quasisymmetric Schur function expansion of $\qs _{\alpha \cskew \beta}$.
\end{corollary}

\begin{proof} If $\qs _{\alpha \cskew \beta}$ is a nonsymmetric skew quasisymmetric Schur function, then by Theorem~\ref{the:uniform} we know that for some $j\in \{2,\ldots,\ell\}$ we have $\alpha _j>\alpha _{j-1}$ or $\alpha _j < \alpha _{j-1}$,
and we choose $j$ to be minimal. By Proposition~\ref{prop:rowfilling} we know that $\qs_\delta$ is a summand in the quasisymmetric Schur function expansion of $\qs _{\alpha \cskew \beta}$. If $\alpha _j>\alpha _{j-1}$ then by the proof of Lemma~\ref{lem:sympart}, specialized to content $\partitionof{\delta}$, the term $\qs _{\partitionof{\delta}}$ does not appear in the expansion. If $\alpha _j<\alpha _{j-1}$ then by the proof of Lemma~\ref{lem:symrevpart} the term $\qs _{\reverse{\partitionof{\delta}}}$ does not appear in the expansion. \end{proof}

\section*{Acknowledgement\cnew{s}}
The second and third authors would like to thank the Institut f{\"u}r Algebra, Zahlentheorie und Diskrete Mathematik of Leibniz Universit\"at for its hospitality and for providing a stimulating venue to conduct part of this research.
\cnew{We also thank the referees for helpful comments and suggestions.}

\end{document}